\documentclass[11pt,a4paper]{article}
\usepackage[dvipsnames,svgnames,table]{xcolor}
\usepackage{sectsty}
\usepackage{multirow}
\usepackage{relsize}
\usepackage[T1]{fontenc} % Use 8-bit encoding that has 256 glyphs
\usepackage{fourier} % Use the Adobe Utopia font for the document - comment this line to return to the LaTeX default
\usepackage[english]{babel} % English language/hyphenation
\usepackage{amsmath,amsfonts,amsthm, amssymb} % Math packages
\usepackage{graphicx, subfigure}
\usepackage{wrapfig}
\usepackage{bm}
\usepackage{soul} % Underline
\usepackage{lipsum} % Used for inserting dummy 'Lorem ipsum' text into the template
\usepackage{sectsty} % Allows customizing section commands
\usepackage{fancyhdr} % Custom headers and footers
\usepackage[T1]{fontenc} % Use 8-bit encoding that has 256 glyphs
\usepackage[english]{babel} % English language/hyphenation
\usepackage{framed}
\usepackage{bm}
\usepackage{fancyhdr}
\usepackage{etoolbox}
\usepackage{titlesec}
\usepackage{titletoc}
\usepackage{float}
\theoremstyle{plain}% default
\newtheorem{theorem}{Theorem}[section]
\newtheorem{lemma}{Lemma}[section]

\theoremstyle{definition}
\newtheorem{definition}{Definition}[section]

\theoremstyle{remark}
\newtheorem{remark}{Remark}

\renewcommand{\v}[1]{\boldsymbol{\mathbf{#1}}} % Black both Greek and Latin letters.
\newcommand{\overbar}[1]{\mkern 1.5mu\overline{\mkern-1.5mu#1\mkern-1.5mu}\mkern 1.5mu}

\renewcommand{\span}{\operatorname{span}}

\let\div\undefined
\DeclareMathOperator{\div}{div}
\DeclareMathOperator{\curl}{curl}
\DeclareMathOperator{\tr}{tr}
\DeclareMathOperator{\ess}{ess}

\numberwithin{equation}{section} % Number equations within sections (i.e. 1.1, 1.2, 2.1, 2.2 instead of 1, 2, 3, 4)
\numberwithin{figure}{section} % Number figures within sections (i.e. 1.1, 1.2, 2.1, 2.2 instead of 1, 2, 3, 4)
\numberwithin{table}{section} % Number tables within sections (i.e. 1.1, 1.2, 2.1, 2.2 instead of 1, 2, 3, 4)

\definecolor{grisf}{rgb}{.47,.47,.47} % barre de droite gris fonce
\newcommand{\colorc}{\color{SkyBlue}}
\newcommand{\colorb}{\color{Cerulean}}
\newcommand{\colora}{\color{Blue}}

\titleformat{\chapter}[display]
  {\normalfont\sffamily\bfseries\huge\colora\centering}{\thechapter}{1ex}
  {{\titlerule[1pt]}\vspace{1.3ex}}[\vspace{1ex}{{\titlerule[1pt]}}]

\titleformat{name=\chapter,numberless}[display]
  {\normalfont\sffamily\bfseries\LARGE\colora\centering}{}{1ex}
  {{\titlerule[1pt]}\vspace{1.3ex}}[\vspace{1ex}{\titlerule[1pt]}\vspace{2ex}]

\titleformat{\section}[hang]{\Large\normalfont\sffamily\bfseries\colorb}{{\thesection\, }}{0 em}
  {}[{\titlerule[1pt]}\vspace{1ex}]

\titleformat{\subsection}[hang]{\Large\normalfont\sffamily\bfseries\colorc}{{\thesubsection\, }}{0 em}
  {}[{\titlerule}\vspace{.7ex}]
\titleformat{\subsubsection}[hang]{\normalfont\sffamily\bfseries\large}{{\thesubsubsection\, }}{0 em}
  {}[{\color{grisf}\titlerule}\vspace{3pt}]
\titleformat{\paragraph}[runin]{\normalfont\sffamily\bfseries\colorc}{}{0 em}
  {\indent}

\makeatletter
\patchcmd{\@fancyhead}{\rlap}{\color{grisf}\rlap}{}{}
\patchcmd{\headrule}{\hrule}{\color{grisf}\hrule}{}{}
\patchcmd{\@fancyfoot}{\rlap}{\color{grisf}\rlap}{}{}
\patchcmd{\footrule}{\hrule}{\color{grisf}\hrule}{}{}
\makeatother

\fancypagestyle{plain}{
  \fancyhead{}
  
  }

\title{An inverse problem for an electroseismic model describing the coupling phenomenon of electromagnetic and seismic waves}
\author{Eric BONNETIER\footnote{Eric BONNETIER,
Institut Fourier, Universit\'e Grenoble-Alpes, 
BP 74, 38402 Saint-Martin-d'H\`eres Cedex, France; Email:  Eric.Bonnetier@univ-grenoble-alpes.fr} \;
 Faouzi TRIKI\footnote{Faouzi Triki,  Laboratoire Jean Kuntzmann,  UMR CNRS 5224, 
Universit\'e  Grenoble-Alpes, 700 Avenue Centrale,
38401 Saint-Martin-d'H\`eres, France; Email: Faouzi.Triki@univ-grenoble-alpes.fr},  \; and   Qi XUE \footnote{Qi Xue,  Laboratoire Jean Kuntzmann,  UMR CNRS 5224, 
Universit\'e  Grenoble-Alpes, 700 Avenue Centrale,
38401 Saint-Martin-d'H\`eres, France; Email: Qi.Xue@univ-grenoble-alpes.fr}}

\begin{document}
\maketitle

\begin{abstract}
The electroseismic model describes the coupling phenomenon of the electromagnetic waves and seismic waves in fluid
immersed porous rock. Electric parameters have better contrast than elastic parameters while seismic waves provide
better resolution because of the short wavelength. The combination of theses two different waves
is prominent in oil exploration. Under some assumptions on the physical parameters, we derived a Hölder stability estimate to the inverse problem of recovery of the electric parameters and the coupling coefficient from the knowledge of the fields in a small open domain near the boundary.
 The proof is based on a Carleman estimate of the electroseismic model. 
\end{abstract}



\pagestyle{fancy}
\section{Introduction}

The traveling of seismic waves underground generates 
electromagnetic (EM) waves and vice versa. This phenomena, electro-kinetic coupling, 
is explained by the electro-kinetic theory, which considers that 
the sediment layers of the earth are porous media saturated with fluid electrolyte. The solid grains
of porous media carry extra electric charges (usually negative) on their surfaces as a result of
the chemical reactions between the ions in the fluid and the crystals that compose the solid.
These charges are balanced by ions of opposite sign in the fluid, forming thus an electrical double layer.
When seismic waves propagate through porous media, the relative solid-fluid motion induces
an electrical current which is a source of EM waves.
Conversely, when EM waves pass through such porous media, ions in the fluid are set in motion and drag
the fluid as well, because of viscous traction.
\vskip .3cm

 Electro-kinetic coupling has been observed by
geophysicist, see e.g., \cite{zhu2000experimental, garambois2001seismoelectric, haines2007seismoelectric}.
This effect rose interest in the physics community, as the coupling of EM and seismic waves may
provide an efficient tool for imaging the subsoil in view of oil prospection.
Such an imaging technique, and the associated inverse problem of reconstructing
the constitutive parameters of the subsoil, fall into the category of multi-physics inverse
problems, where a medium is probed using two types of waves (see for example~\cite{ammari2017multi, kuchment2011mathematics}
and references therein for medical imaging). 
One type of waves is very
sensitive to the contrast in the parameters that describe the properties of the medium
(electric permittivity, magnetic permeability and conductivity in our case) 
however, these waves are usually very diffusive and
only scattered information arrives to the medium boundary, where the data are collected. 
The other type, on the contrary, is not very sensitive to changes in the medium
properties, but is able to carry information through the medium with little distortion (seismic waves in our case).
\vskip .3cm

In 1994, Pride \cite{pride1994governing} derived a macroscopic model in the frequency domain
that models the coupling of EM and seismic waves in fluid-saturated porous media by averaging microscopic
properties, see also \cite{pride1996electroseismic, pride2005electroseismic}.
The associated system of equations is composed of the Maxwell equations, which govern the propagation of EM waves,
and of the Biot equations \cite{biot1962generalized}, which govern the propagation of seismic waves in porous media.
\vskip .3cm

Because the electro-kinetic coupling is very weak in practice, one usually neglects
multi-conversion, i.e., one neglects the coupling terms in either the Maxwell or the Biot equations, and thus 
only considers transformations from either EM to seismic waves (\textit{electroseismic}) or from seismic to EM waves (\textit{seismoelectric}).
At low frequency, one can expand all the parameters of the model with respect to frequency, and neglecting high order terms
results in a time domain model. This is done for example in \cite{haines2006seismoelectric} for
the seismoelectric model. In our paper, we are interested in the electroseismic model, which takes the form
\begin{eqnarray}
\partial_t \v D-\curl (\alpha\v B) +\gamma \v D &=& \v 0 ,\label{em_d}\\
\partial_t \v B+\curl (\beta \v D) &=& \v 0 , \label{em_b}\\
\rho  \partial^2_t\v u+\rho_f\partial^2_t\v w-\div \v \tau &=& \v 0 \label{biot_u} ,\\
\rho_f\partial^2_t\v u+\rho_e\partial^2_t\v w+\nabla p    
	+\tfrac{\eta}{\kappa}\partial_t \v w-\xi\v D &=& \v 0 \label{biot_w} ,\\
(\lambda\div \v u+C\div \v w)\v I+G(\nabla\v u+\nabla\v u^T) &=& \v \tau \label{biot_tau} ,\\
C\div \v u+ M\div \v w &=& -p \label{biot_p},
\end{eqnarray}
where 
$$\alpha = \tfrac{1}{\mu},\ \ \beta = \tfrac{1}{\varepsilon},\ \  \gamma = \tfrac{\sigma}{\varepsilon},
\ \ \xi = \tfrac{L\eta}{\kappa\varepsilon}.$$
The physical meaning of all the variables and parameters is given in Table \ref{table_param}.
All the parameters are real and positive. 
Throughout the text, we denote by $\partial_t$ and $\partial_j$ the partial derivatives of a function
with respect to $t$ and $x_j$ respectively, and by $\nabla$ (resp. $\nabla_{\v x,t}$) the gradients with respect
to the variables $\v x$ (resp. $\v x$ and t). By gradient of a vector-valued function, we mean the transpose of the Jacobian
matrix. 

\begin{table}
\label{table_param}
\centering
\begin{tabular}{rlrl}
\hline
$\v{D}$ & electric flux &
$\v{B}$ & magnetic flux \\
$\v{u}$ & solid displacement &
$\v{w}$ & relative fluid displacement \\
$\v\tau$ & bulk stress tensor &
$p     $ & pore pressure \\
$\sigma$ & electric conductivity &
$\varepsilon$ & electric permittivity \\
$\mu$         & magnetic permeability &
$L$      & electro-kinetic parameter\\
$\kappa$ & fluid flow permeability &
$\eta$   & fluid viscosity \\
$\lambda, G$  & Lam\'{e} elastic parameters &
$C, M$   & Biot moduli parameters \\
$\rho  $ & bulk density &
$\rho_f$ & fluid density \\
$\rho_e$ & equivalent density \\
\hline
\end{tabular}
\caption{Physical meanings of variables and parameters}
\end{table}

\vskip .3cm
To close the Maxwell system (\ref{em_d})-(\ref{em_b}), 
we assume that the media do not contain any free charge, i.e., 
\begin{equation}
\div\v D=\div\v B=0. \label{em_div}
\end{equation}
We consider the system of equations (\ref{em_d})-(\ref{em_div}) in $Q=\Omega\times(-T,T)$ where $\Omega\in\mathbb{R}^3$
is a bounded domain with $C^{\infty}$ boundary $\partial\Omega$.
The boundary conditions are
\begin{equation}
\v n\times\v D=\v 0,\ \ \v n\cdot\v B=0,\ \  \v n\cdot \v \tau=\v 0,\ \  p=0, \quad\text{ on } \partial\Omega,
\end{equation}
where $\v n$ is the outer normal vector on $\partial\Omega$.
As we are interested in the electroseismic model, we consider 0 initial values for the solid displacement $\v u$ and for the
relative fluid displacement $\v w$ associated to the Biot equations
\begin{equation}
\v u(\v x,0) = \v 0,\ \ \v w(\v x,0) = \v 0,\ \ \partial_t \v u(\v x,0) = \v 0,\ \ \partial_t\v w(\v x,0) = \v 0,
\end{equation}
while we impose electric and magnetic fluxes in $\Omega$
\begin{equation}
\v D(\v x,0) = \v D_0(\v x),\ \ \v B(\v x,0) = \v B_0(\v x).
\end{equation}

\vskip .3cm

In accordance with the accepted physical properties of underground media, we assume that the matrices
\begin{equation}\label{equation:spd}	
\left(\begin{array}{cc}
\rho & \rho_f\\ \rho_f & \rho_e
\end{array}\right)\quad\text{and}\quad\left(\begin{array}{cc}
\lambda & C\\ C & M
\end{array}\right)\quad
\end{equation}
are symmetric positive definite and that
$\rho_e>\rho_f, \rho>\rho_f$.
We also assume that all the parameters of the Biot equations (\ref{biot_u})-(\ref{biot_p}) are known, 
except the coupling coefficient $\xi$.
The main object of this paper is to analyse  the well-posedness of the  inverse problem of determining the parameters $(\alpha,\beta,\gamma,\xi)$
 from measurements of $(\v D, \v B, \v u, \v w)$ in $Q_\omega$, where $Q_\omega=\omega\times(-T,T)$ and $\omega\subset \Omega$ is 
a fixed neighborhood of the boundary. To the best of our knowledge, \cite{chen2013inverse} and its following work
\cite{chen2014inverse} are the only papers considering the inverse electroseismic problem. 
In those papers, the authors considered the second inversion step in 
frequency domain, 
assuming  that $L\v E$ is known everywhere in $\Omega$, they focus on the identification of $(L,\sigma)$.
Their method is based on the CGO solutions of frequency domain Maxwell equations \cite{colton1992uniqueness,
sylvester1987global}.
Different from their work, our method treats the  global inversion and  is based on a  Carleman estimate to the electroseismic  model \cite{bellassoued2013carleman, 0266-5611-8-4-009, isakov2004carleman}.
Under some assumptions on the physical parameters 
we derive a H\"older  stability estimate  to the inverse problem of  identification of the electric parameters and the coupling coefficient with only measurements near the boundary.  The main stability result is provided in 
Theorem~\eqref{theorem:inverse_problem}.
\vskip .3cm

The paper is organized as follows:   Section \ref{section:forward} is devoted to the existence and uniqueness 
of solutions to the forward problem.
In Section \ref{section:carleman}, we derive a  Carleman estimate for the whole electroseismic system, from which 
we infer, in section \ref{section:inverse_problem}, the H\"{o}lder stability of the inverse problem with measurements 
of $(\v D, \v B, \v u, \v w)$ near the boundary.
%%%%%%%%%%%%%%%%%%%%%%%%%%%%%%%%%%%%%%%%%%%%%%%%

\section{Existence and uniqueness for Biot's system}\label{section:forward}
As stated before, in the electroseismic system the Maxwell equations are totally independent of 
the Biot equations. Therefore, the question of existence and uniqueness of solutions to the electroseismic system
reduces to showing existence and uniqueness of solutions to the Biot equations.
To the author's best knowledge, existence and uniqueness for the Biot equations in two dimension was first proved in
\cite{santos1986elastic}. In \cite{bellassoued2013carleman}, the 3D case is studied, but with different boundary conditions
than those considered here. Although the general arguments are similar, we prove the existence and uniqueness of solutions to our version of the Biot equations for the sake of completeness.

\vskip .3cm
We first introduce some notations.
For two matrices $\v E=(E_{ij}), \v F=(F_{ij})$ of the same size, we define 
$$\v E:\v F = \sum_{i,j}E_{ij}F_{ij}.$$
For a given Hilbert space H, $(u,v)_H$ denotes the inner product of $u,v\in H$ and $\|u\|_H$ the corresponding
norm. The dual space of $H$ is denoted by $H'$ and $\langle u, f\rangle$ represents the duality pairing of
$u\in H, f\in H'$. We use $[H]^{m}$ to denote the space of vector-valued functions $\v u=(u_1,\ldots,u_m)$ such that
$u_j\in H,\ 1\le j\le m$. The inner product on this space is defined by
$(\v u, \v v)_{H} = \sum_i(u_i,v_i)_H$. We use similar definitions for spaces of matrix-valued functions.
When we consider the space $L^2(\Omega)$ or $[L^2(\Omega)]^m$, we usually omit all subscripts. 
The Sobolev space $H(\div,\Omega)$ is defined by 
$$\big\{\v u\in[L^2(\Omega)]^3:\div\v u\in L^2(\Omega)\big\}$$
and is equipped with the inner product
$$(\v u, \v v)_{H(\div)}=(\v u,\v v)+(\div\v u,\div\v v).$$
We use $L^p(-T,T;H)$ to denote the space of functions $f:(-T, T)\rightarrow H$ satisfying
$$\|f\|_{L^p(-T,T;H)}:=\left(\int_{-T}^{T}\|f\|_H^pdt\right)^{1/p}<\infty$$ 
for $1\le p< \infty$, and define
$$\|f\|_{L^{\infty}(-T,T;H)}:=\ess\sup_{t\in (-T, T)}\|f\|_H<\infty.$$
Denoting $V=[H^1(\Omega)]^3\times H(\div,\Omega), \v v_1 = \v u, \v v_2 = \v w$, 
$$\v v = \left(\begin{array}{c}
\v v_1 \\ \v v_2
\end{array}\right),\quad
\v F = \left(\begin{array}{c}
\v 0 \\ \xi\v D
\end{array}\right),$$
$$\v A=\left(\begin{array}{cc}
\rho  \v I_3 & \rho_f \v I_3 \\
\rho_f\v I_3 & \rho_e \v I_3
\end{array}\right),\quad 
\v B=\left(\begin{array}{cc}
\v 0 & \v 0 \\
\v 0 & \tfrac{\eta}{\kappa}\v I_3
\end{array}\right),\quad
\mathcal{L}\v v = \left(\begin{array}{c}
-\div\v \tau \\
\nabla p 
\end{array}\right),
$$
the Biot equations can be compactly written in the form 
\begin{equation}\label{equation:biot_compact}
\left\{\begin{array}{cl}
\v A\partial_t^2\v v+\v B\partial_t\v v+\mathcal{L}\v v = \v F, &\quad \text{in } \Omega\times(-T,T),\\
\v v(\v x, 0) = 0, & \quad \text{in } \Omega,\\
\partial_t\v v(\v x, 0) = 0, & \quad \text{in } \Omega,\\
\v n\cdot \v \tau=\v 0,\ \  p=0, &\quad\text{on } \partial\Omega\times(-T,T).
\end{array}\right.
\end{equation}
Integration by parts and using the boundary conditions, we have
\begin{eqnarray*}
(\mathcal{L}\v v,\v v') &=& \int_{\Omega}\big(-\div\v\tau\cdot\v v_1'+\nabla p\cdot\v v_2'\big)\\
&=& \int_{\Omega}\big(\v\tau : \nabla\v v_1'-p\div\v v_2' \big)\\
&=& \big(\div\v v_1,\lambda\div\v v_1'+C\div\v v_2'\big)+
 	\big(\div\v v_2,C\div\v v_1'+M\div\v v_2'\big)+\big(2G e(\v v_1), e(\v v_1')\big),
\end{eqnarray*}
where $e(\v v_1)=\tfrac{1}{2}(\nabla\v v_1+\nabla\v v_1^T).$\\

Define 
$$\mathcal{B}(\v v,\v v') = \big(\div\v v_1,\lambda\div\v v_1'+C\div\v v_2'\big)+
 	\big(\div\v v_2,C\div\v v_1'+M\div\v v_2'\big)+\big(2G e(\v v_1), e(\v v_1')\big).$$
It's obvious that $\mathcal{B}$ is a symmetric bounded bilinear form.
We recall the Korn inequality
$$\big(e(\v v_1),e(\v v_1)\big)\ge C_0\|\v v_1\|^2_{H^1}-\|\v v_1\|^2,$$
where $C_0$ is a  strictly positive constant. From now on, we use $C_0$ to denote a general positive constant
which may take different values at different places.
From the Korn inequality, we obtain
\begin{eqnarray*}
\mathcal{B}(\v v,\v v) &\ge& \int_{\Omega}\left(\div\v v_1\ \ \div\v v_2\right)
	\left(\begin{array}{cc}
		\lambda & C\\ C & M
		\end{array}\right)\left(\begin{array}{c}
			\div\v v_1\\ \div\v v_2
			\end{array}\right)d\v x+
	2\min\{G\}\big(e(\v v_1),e(\v v_1)\big)\\
&\ge& \lambda_* \|\div\v v_1\|^2+\lambda_*\|\div \v v_2\|^2+2C_0\min\{G\}\|\v v_1\|^2_{H^1}-2\min\{G\}\|\v v_1\|^2\\
&\ge& C_0\|\v v\|^2_V-\theta\|\v v\|^2,
\end{eqnarray*}
where $\theta$ is a positive constant independent of $\v v$ and $\lambda_*$ is the smallest eigenvalue of the matrix 
$$\left(\begin{array}{cc}\lambda & C\\ C & M  \end{array}\right).$$
We define $\mathcal{B}_{\theta}(\v v,\v v')=\mathcal{B}(\v v,\v v')+\theta(\v v,\v v')$. The bilinear form $\mathcal{B}_{\theta}$ is
symmetric, bounded, and it satisfies the following ellipticity condition $\mathcal{B}_{\theta}(\v v, \v v)\ge C_0\|\v v\|_V^2$.
\begin{definition}Let $\ \v F\in H^1(-T,T;[L^2(\Omega)]^6)$.
We call $\v r\in L^{\infty}(-T,T;V)$ a generalized solution to (\ref{equation:biot_compact}) if it satisfies
\begin{equation}\label{equation:biot_generalized}
(\v A\partial_t^2\v r(t),\v v)+(\v B\partial_t\v r(t), \v v )+\mathcal{B}(\v r(t), \v v) = (\v F(t), \v v)
\quad \text{ a.e. } t\in(-T,T)
\end{equation}
for any $\v v\in V$.
\end{definition}
Note that the scalar products in this definition only involve the $\v x$ variable.
We can now state the existence and uniqueness theorem for the Biot equations.
%%%%%%%%%%%%%%%%%%%%%%%%%%
\begin{theorem}
Let $\ \v F\in H^1(-T,T;[L^2(\Omega)]^6)$. Then the system
(\ref{equation:biot_compact}) has a unique weak solution $\v r(\v x, t)$ such that 
$$\v r,\ \partial_t \v r\in L^{\infty}(-T,T;V(\Omega)),\text{ and } \partial_t^2\v r\in L^{\infty}(-T,T;[L^2(\Omega)]^6).$$
\end{theorem}
%%%%%%%%%%%%%%%%%%%
\begin{proof}
Since $V$ is separable, there exists a sequence of linearly independent functions  $\{\v v^{(n)}\}_{n\ge 1}$ 
which form a basis of $V$. Let us define
$$S_m = \span\left\{\v v^{(1)}, \v v^{(2)}, \ldots, \v v^{(m)}\right\},$$
and choose 
$$\v r^{(m)}(t) = \sum_{j=1}^m g_{jm}(t)\v v^{(j)}$$
such that $\v r^{(m)}(0)\rightarrow 0, \ \partial_t \v r^{(m)}(0)\rightarrow 0$.
The functions $g_{jm}(t)$ are determined by the system of ordinary differential equations 
\begin{equation}\label{equation:biot_generalized1}
\big(\v A\partial_t^2\v r^{(m)},\v v\big)+\big(\v B\partial_t\v r^{(m)}, \v v\big) + \mathcal{B}\big(\v r^{(m)},\v v\big) 
= (\v F, \v v), \quad \v v\in S_m.
\end{equation}
Next we prove two a priori estimates of $\v r^{(m)}(t)$. By choosing $\v v=\partial_t\v r^{(m)}$, we obtain
\begin{equation}\label{equation:biot_generalized_approximate}
\big(\v A\partial_t^2\v r^{(m)}, \partial_t\v r^{(m)}\big)+\big(\v B\partial_t\v r^{(m)}, \partial_t\v r^{(m)}\big) + 
\mathcal{B}\big(\v r^{(m)},\partial_t\v r^{(m)}\big) = \big(\v F, \partial_t\v r^{(m)}\big).
\end{equation}
Let $\Lambda(t) = \|\v A^{1/2}\partial_t\v r^{(m)}(t)\|^2+\mathcal{B}_{\theta}\big(\v r^{(m)}(t),\v r^{(m)}(t)\big)$.
Since $\mathcal{B}_{\theta}$ is elliptic, $\Lambda(t)$ can be lower bounded by
$$\Lambda(t)\ge C_0\big(\|\v r^{(m)}(t)\|^2_V+\|\partial_t\v r^{(m)}(t)\|^2\big)$$
and from (\ref{equation:biot_generalized_approximate})
$$\frac{d}{dt}\Lambda(t) \le C_0\big(\|\v F(t)\|^2+\|\v r^{(m)}(t)\|_V^2+\|\partial_t\v r^{(m)}(t)\|^2\big).$$
Integrating from $0$ to $t$ yields
$$\Lambda(t)\le C_0\int_{-T}^T\|\v F(\tau)\|^2 d\tau+
	\Lambda(0)+C_0\int_0^t\big(\|\v r^{(m)}(\tau)\|_V^2+\|\partial_{\tau}\v r^{(m)}(\tau)\|^2\big) d\tau.$$
Since $\Lambda(0) = \|\v A^{1/2}\partial_t\v r^{(m)}(0)\|^2+\mathcal{B}_{\theta}\big(\v r^{(m)}(0),\v r^{(m)}(0)\big)$
and $\v r^{(m)}(0), \partial_t \v r^{(m)}(0)\rightarrow 0$, $\Lambda(0)$ is bounded by a constant $C_0$ independent of $m$.
We conclude that
\begin{equation}
\|\v r^{(m)}(t)\|_V^2+\|\partial_t\v r^{(m)}(t)\|^2 \le
C_0+C_0\int_0^t\big(\|\v r^{(m)}(\tau)\|^2+\|\partial_{\tau}\v r^{(m)}(\tau)\|^2\big)
\end{equation}
and by the Gronwall inequality
\begin{equation}\label{equation:biot_gronwall1}
\|\v r^{(m)}(t)\|_V^2+\|\partial_t\v r^{(m)}(t)\|^2\le C_0
\end{equation}
where $C_0$ is independent of $t$ and $m$.
Taking the time derivative of (\ref{equation:biot_generalized1}) and choosing $\v v = \partial_t^2\v r^{(m)}$, we have
\begin{equation}
\big(\v A\partial_t^3\v r^{(m)}, \partial_t^2\v r^{(m)}\big)+\big(\v B\partial_t^2\v r^{(m)}, \partial_t^2\v r^{(m)}\big) + 
\mathcal{B}\big(\partial_t\v r^{(m)},\partial_t^2\v r^{(m)}\big) = \big(\partial_t\v F, \partial_t^2\v r^{(m)}\big).
\end{equation}
Following the same process that leads to (\ref{equation:biot_gronwall1}), we obtain
\begin{equation}\label{equation:biot_gronwall2}
\|\partial_t\v r^{(m)}(t)\|_V^2+\|\partial_t^2\v r^{(m)}(t)\|^2\le 
C_0.
\end{equation}Therefore 
$$\v r^{(m)},\ \partial_t\v r^{(m)}\in L^{\infty}(-T,T;V),\quad \partial_t^2\v r^{(m)}\in L^{\infty}(-T,T;[L^2(\Omega)]^6)$$
are bounded.  It follows that we can extract a subsequence of $\{\v r^{(m)}\}$, still denoted by $\{\v r^{(m)}\}$, such that
$$\v r^{(m)}\rightarrow\v r, \ \partial_t\v r^{(m)}\rightarrow \partial_t\v r\quad \text{ weak-* in } L^{\infty}(-T,T; V)$$
and
$$\partial_t^2\v r^{(m)}\rightarrow \partial_t^2\v r\quad \text{ weak-* in } L^{\infty}(-T,T; [L^2(\Omega)]^6).$$
Since $\{\v v^{(m)}\}$ is dense in $V$, we have, for any $\v v\in V$,
\begin{eqnarray*}
\big(\v A\partial_t^2\v r^{(m)},\v v\big)\rightarrow  \big(\v A\partial_t^2\v r,\v v\big) & & \text{ weak-* in } L^{\infty}(-T,T), \\ 
\big(\v B\partial_t\v r^{(m)}, \v v\big)\rightarrow  \big(\v B\partial_t\v r, \v v\big) & & \text{ weak-* in } L^{\infty}(-T,T), \\
\mathcal{B}\big(\v r^{(m)},\v v\big)\rightarrow  \mathcal{B}\big(\v r^{(m)},\v v\big) & & \text{ weak-* in } L^{\infty}(-T,T).
\end{eqnarray*}
The existence is completed by letting $m\rightarrow \infty$ in (\ref{equation:biot_generalized1}).
The uniqueness is obvious from (\ref{equation:biot_gronwall1}) and (\ref{equation:biot_gronwall2})
by choosing $\v F = 0$.
\end{proof}

%%%%%%%%%%%%%%%%%%%%%%%%%%%%%%%%%%%%%%%%%%%%%%
\section{A Carleman estimate for the electroseismic model}\label{section:carleman}
To derive a Carleman estimate for a system of equations, the usual process consists in diagonalizing the system
and then in applying a Carleman estimate for each scalar equation that composes the diagonalized system
\cite{isakov2006inverse}. 
We first recall a known Carleman estimate for the scalar wave equation \cite{isakov2004carleman, isakov2006inverse}.
\begin{lemma}\label{lemma:scalar_wave_carleman}
Assume that there exists a point $\v x_*\in \mathbb{R}^3\setminus\overbar{\Omega}$ and a strictly positive function  $c(\v x)\in C^1(\overbar\Omega)$ which satisfies 
\begin{equation}\label{equation:pseudoconvex}
\frac{\nabla c\cdot (\v x-\v x_*)}{2c} <1-c_0,\quad \text{for all }  \v x\in\overbar{\Omega} ,
\end{equation}
where $c_0\in (0,1)$ is a fixed  constant. Then, there exist constants $\varsigma, \theta, C_0>0$, such that the function
$\varphi = e^{\theta\psi}$ given by $\psi = |\v x-\v x_*|^2-\varsigma|t|^2$
satisfies $\varphi(\v x,T)=\varphi(\v x,-T)<1, \varphi(\v x,0)\ge 1$ and
$$\int_Q e^{2\tau\varphi}\big(\tau^3|u|^2+\tau|\nabla_{\v x, t} u|^2\big)\leq C_0\int_Q e^{2\tau\varphi}|f|^2,$$
for all $\tau$ large than a positive constant $\tau_0$ and for any $u\in C_0^2(Q)$ that solves
$$\partial_t^2 u-c(\v x)\Delta u = f.$$
The notation $|\cdot|^2$ means the sum of the square of all the components of vectors or matrices.
\end{lemma}
\begin{remark}\label{remark:pseudo_convex}
For any  $\epsilon>0$ sufficiently small, there exists a constant $\delta$ such that $\varphi(\v x,t)>1-\epsilon$ for $|t|<\delta$ and 
$\varphi(\v x, t)<1-2\epsilon$ for $t>T-\delta$ or $t<-T+\delta$. We denote 
$$\varphi_0(\v x)=\varphi(\v x,0),\quad  \Phi=\max_{(\v x,t)\in Q}\varphi.$$
\end{remark}

For the Maxwell equations with $\sigma=0$, Carleman estimates can be found, for example, 
in \cite{isakov2006inverse, bellassoued2012inverse}.
The arguments in these references easily generalize to the case $\sigma\neq 0$.
\begin{lemma}\label{lemma:em_carleman}
Assume that $\alpha, \beta \in C^2(\overbar\Omega)$ and $\gamma \in C^1(\overbar{\Omega})$, such that
$\alpha,\beta >\alpha_0 >0$ and $\gamma \ge 0$. Assume additionally that
the wave speed $c:=\alpha\beta$ satisfies (\ref{equation:pseudoconvex}).
Then there exists a constant $C_0$ such that 
$$\int_Q e^{2\tau\varphi}\Big(\tau^3\big(|\v D|^2+|\v B|^2\big)+
                              \tau  \big(|\nabla_{\v x, t}\v D|^2+|\nabla_{\v x, t}\v B|^2\big)\Big)
\leq C_0\int_Q e^{2\tau\varphi}\big(|\v J_1|^2+|\v J_2|^2+|\nabla_{\v x, t}\v J_1|^2+|\nabla_{\v x, t}\v J_2|^2\big),$$
for all $\tau$ larger than a positive constant $\tau_0$ and for any $\v D, \v B\in C_0^2(Q)$ that solve
\begin{equation}
\left\{\begin{array}{lcc}
\partial_t \v D-\curl (\alpha\v B) +\gamma \v D &=& \v J_1 , \label{equation:em1}\\
\partial_t \v B+\curl (\beta \v D) &=& \v J_2 , \\
\div \v D=\div \v B = 0. & &
\end{array}\right.
\end{equation}
\end{lemma}
\begin{proof}
By substitution, the system can be transformed into the following two equations
\begin{eqnarray*}
\partial_t^2 \v D-\alpha\beta\Delta \v D &=& \partial_t\v J_1+\curl(\alpha\v J_2)-\mathcal{R}_1, \\
\partial_t^2 \v B-\alpha\beta\Delta \v B &=& \partial_t\v J_2-\curl(\beta \v J_1)-\mathcal{R}_2,
\end{eqnarray*}
where 
\begin{eqnarray*}
\mathcal{R}_1 &=& \nabla(\alpha\beta)\times\curl\v D+\curl(\alpha\nabla\beta\times\v D)+\gamma\partial_t \v D ,\\
\mathcal{R}_2 &=& \nabla(\alpha\beta)\times\curl\v B+\curl(\beta\nabla\alpha\times\v B)-\curl(\beta\gamma\v D) .
\end{eqnarray*}
Applying Lemma \ref{lemma:scalar_wave_carleman} to each component of the equations, we have 
$$\int_Q e^{2\tau\varphi}\Big(\tau^3|\v D|^2+\tau|\nabla_{\v x, t}\v D|^2\Big)
\leq C_0\int_Q e^{2\tau\varphi}\big(F+|\v D|^2+|\nabla_{\v x, t}\v D|^2\big),$$
$$\int_Q e^{2\tau\varphi}\Big(\tau^3|\v B|^2+\tau|\nabla_{\v x, t}\v B|^2\Big)
\leq C_0\int_Q e^{2\tau\varphi}\big(F+|\v D|^2+|\nabla_{\v x, t}\v D|^2+|\v B|^2+|\nabla_{\v x, t}\v B|^2\big),$$
where $F = |\v J_1|^2+|\v J_2|^2+|\nabla_{\v x, t}\v J_1|^2+|\nabla_{\v x, t}\v J_2|^2$.
Adding these two inequalities and taking $\tau$ large enough to absorb the right hand side terms completes the proof.
\end{proof}

Before deriving a Carleman estimate for the Biot equations, we study the property of the associated matrix of material coefficients.
Define
\begin{equation}\label{equation:a_c}
\rho_0 = \rho\rho_e-\rho_f^2,\quad 
\v a = \left(\begin{array}{cc}
a_{11} & a_{12}\\ a_{21} & a_{22}
\end{array}\right)
=\left(\begin{array}{cc}
\tfrac{\rho_0}{\rho_e} & \rho_f \\
                     0 & \rho_e
\end{array}\right)^{-1}
\left(\begin{array}{cc}
\lambda+G-\tfrac{\rho_f}{\rho_e}C   & C \\ C-\tfrac{\rho_f}{\rho_e}M & M 
\end{array}\right),\quad
c=\tfrac{\rho_e}{\rho_0}G.
\end{equation}
From the positive definite of the matrices (\ref{equation:spd}), we have $\rho_0>0$.
Let us denote
\begin{equation}\label{equation:a_tilde}
\tilde{\v a} = \left(\begin{array}{cc}
c+a_{11} & a_{12}\\ a_{21} & a_{22}
\end{array}\right)
\end{equation}
which can be expanded into
$$\tilde{\v a} = 
\left(\begin{array}{cc}
\tfrac{\rho_e}{\rho_0}(\lambda+2G)-2\tfrac{\rho_f}{\rho_0}C+\tfrac{\rho_f^2}{\rho_0\rho_e}M   & 
\tfrac{\rho_e}{\rho_0}C-\tfrac{\rho_f}{\rho_0}M \\ \tfrac{1}{\rho_e}C-\tfrac{\rho_f}{\rho_e^2}M & \tfrac{1}{\rho_e}M 
\end{array}\right).$$
The two eigenvalues of $\tilde{\v a}$ are 
$$\frac{(c+a_{11}+a_{22})\pm \sqrt{(c+a_{11}-a_{22})^2+4a_{12}a_{21}}}{2}.$$
Since 
$$a_{12}a_{21}=\tfrac{1}{\rho_0}\big(C-\tfrac{\rho_f}{\rho_e}M\big)^2\ge 0,$$
$\tilde{\v a}$ has two real eigenvalues.
The determinant of $\tilde{\v a}$ is
$$\det \tilde{\v a} = \tfrac{1}{\rho_0}\big((\lambda+2G)+\tfrac{\rho_f^2}{\rho_e^2}M-2\tfrac{\rho_f}{\rho_e}C\big)M
-\tfrac{1}{\rho_0}\big(C-\tfrac{\rho_f}{\rho_e}M\big)^2 = \tfrac{1}{\rho_0}\big(\lambda M-C^2+2GM\big)>0,$$
and its trace is
$$\tr \tilde{\v a} =\tfrac{1}{\rho_0}\big(\rho_e(\lambda+2G)+\rho M-2\rho_f C\big)\ge
\tfrac{\rho_f}{\rho_0}\big(\lambda+2G+M-2C\big).$$
From the positive definite of the matrices (\ref{equation:spd}), we have $(\lambda+M)^2 \ge 4\lambda M > 4C^2$ and hence
$\tr \tilde{\v a}>0$.
Therefore $\tilde{\v a}$ is similar to a diagonal matrix and it has two positive eigenvalues.

\vskip .3cm

In the following, we will derive a Carleman estimate for the Biot system (\ref{biot_u})-(\ref{biot_p}).
The idea is similar as with the Maxwell or the elastic system. We emphasize that the results from
\cite{bellassoued2013carleman} do not apply directly to our Biot system, which is different from the one treated in
that reference due to the presence of the term $\partial_t \v w$. 
We will explain in detail the difference in the proof of the following lemma.

\begin{lemma}\label{lemma:biot_carleman}
Assume that all the parameters in the Biot equations are in $C^3(\overbar{\Omega})$. 
Assume that $c=\tfrac{\rho_e}{\rho_0}G$ and two eigenvalues of the matrix $\tilde{\v a}$ given by (\ref{equation:a_tilde})
satisfy the condition (\ref{equation:pseudoconvex}). Then there exists a constant $C_0$ such that 
\begin{eqnarray*}
 &     & \int_Q e^{2\tau\varphi}\Big(\tau^3\big(|\v u|^2+|\div\v u|^2+|\div\v w|^2+|\curl\v u|^2\big)\\
 &	   &\quad+ \tau  \big(|\nabla_{\v x, t}\v u|^2+|\nabla_{\v x, t}(\div\v u)|^2+
 	|\nabla_{\v x, t}(\div\v w)|^2+|\nabla_{\v x, t}(\curl\v u)|^2\big)\Big) \\
 &\leq & C_0\int_Q e^{2\tau\varphi}\big(|\v F_1|^2+|\v F_2|^2+|\v D|^2+|\nabla\v F_1|^2+|\nabla\v F_2|^2+|\nabla\v D|^2\big),
\end{eqnarray*}
for all $\tau$ larger than a positive constant $\tau_0$ and for any $\v u, \v w\in C_0^3(Q)$ that solve
\begin{equation}\label{equation:biot_p1}
\left\{\begin{array}{lcc}
\rho  \partial^2_t\v u+\rho_f\partial^2_t\v w-\div \v \tau &=& \v F_1 ,\\
\rho_f\partial^2_t\v u+\rho_e\partial^2_t\v w+\nabla p    
	+\tfrac{\eta}{\kappa}\partial_t \v w-\xi\v D &=& \v F_2 ,\\
(\lambda\div \v u+C\div \v w)\v I+G(\nabla\v u+\nabla\v u^T) &=& \v \tau ,\\
C\div \v u+ M\div \v w &=& -p .
\end{array}\right.
\end{equation}
\end{lemma}

\begin{proof}
Let $\v v = \v w+\tfrac{\rho_f}{\rho_e}\v u$
and replace $\v w$ by $\v u, \v v$ in the above system, to obtain
\begin{eqnarray*}
\tfrac{\rho_0}{\rho_e}\partial^2_t\v u+\rho_f\partial^2_t\v v-\div \v \tau &=& \v F_1 ,\\
\rho_e\partial^2_t\v v+\nabla p+\tfrac{\eta}{\kappa}\left(\partial_t \v v-\tfrac{\rho_f}{\rho_e}\partial_t\v u\right)
	-\xi\v D &=& \v F_2, \\
\Big(\big(\lambda-C\tfrac{\rho_f}{\rho_e}\big)\div \v u+C\div \v v
	-C\v u\cdot \nabla\tfrac{\rho_f}{\rho_e}\Big)\v I+G(\nabla\v u+\nabla\v u^T) &=& \v \tau ,\\
\big(C-M\tfrac{\rho_f}{\rho_e}\big)\div \v u+ M\div \v v - M\v u\cdot \nabla\tfrac{\rho_f}{\rho_e} &=& -p,
\end{eqnarray*}
where $\v I$ is the identity matrix of order 3.
After substitution of $\v \tau$ and $p$, we have
\begin{eqnarray}
\tfrac{\rho_0}{\rho_e}\partial^2_t\v u+\rho_f\partial^2_t\v v-G\Delta\v u-
	\left(\lambda+G-\tfrac{\rho_f}{\rho_e}C\right)\nabla\div\v u-C\nabla\div\v v 
	&=& \v F_1 +\mathcal{P}_1 \label{biot_u1},\\
\rho_e\partial^2_t\v v+\tfrac{\eta}{\kappa}\partial_t\v v-
	\left(C-\tfrac{\rho_f}{\rho_e}M\right)\nabla\div\v u-M\nabla\div\v v
	&=& \v F_2 +\mathcal{P}_2 \label{biot_w1},
\end{eqnarray}
where 
\begin{eqnarray*}
\mathcal{P}_1 &=& (\div\v u)\nabla\left(\lambda-\tfrac{\rho_f}{\rho_e}M\right)+
	(\div\v v)\nabla C+(\nabla\v u+\nabla\v u^T)\cdot\nabla G-\nabla\left(C\v u\cdot \nabla\tfrac{\rho_f}{\rho_e}\right),\\
\mathcal{P}_2 &=& (\div\v u)\nabla\left(C-\tfrac{\rho_f}{\rho_e}M\right)      +
	(\div\v v)\nabla M-\nabla\left(M\v u\cdot \nabla\tfrac{\rho_f}{\rho_e}\right)+\tfrac{\rho_f\eta}{\rho_e\kappa}\partial_t\v u+\xi\v D.
\end{eqnarray*}
Set $r=\div\v u, s=\div\v v, \v m=\curl \v u, \v n = \curl\v v$ and
$$\v K=\left(\begin{array}{cc}
\tfrac{\rho_0}{\rho_e}\v I & \rho_f \v I \\
       \v 0 & \rho_e \v I 
\end{array}\right)^{-1}.$$
We multiply the equation system (\ref{biot_u1})-(\ref{biot_w1}) by $\v K$ to obtain
\begin{eqnarray}
\partial^2_t\v u - c\Delta\v u -\tfrac{\rho_f\eta}{\rho_e\kappa}\partial_t\v v-a_{11}\nabla\div\v u -a_{12}\nabla\div\v v
	&=& \v G_1+\mathcal{P}_3 \label{biot_u3},\\
\partial^2_t\v v +\tfrac{\eta}{\rho_e\kappa}\partial_t\v v-a_{21}\nabla\div\v u -a_{22}\nabla\div\v v
	&=& \v G_2+\mathcal{P}_4 \label{biot_w3},
\end{eqnarray}
where 
$$c=\tfrac{\rho_e}{\rho_0}G,\ 
\left(\begin{array}{c}\v G_1\\ \v G_2\end{array}\right)=\v K  \left(\begin{array}{c}\v F_1\\ \v F_2\end{array}\right),\
\left(\begin{array}{c}\mathcal{P}_3 \\ \mathcal{P}_4\end{array}\right)=
	\v K \left(\begin{array}{c}\mathcal{P}_1 \\ \mathcal{P}_2\end{array}\right)$$
and $\v a$ is given by (\ref{equation:a_c}).
Note that $\mathcal{P}_3$ and $\mathcal{P}_4$ are composed of $r, s, \v u, \nabla_{\v x,t}\v u, \v D$.
The equations (\ref{biot_u3}) and (\ref{biot_w3}) can be rewritten as
\begin{eqnarray}
\partial^2_t\v u - c\Delta\v u 	&=& \v G_1+\mathcal{Q}_1 \label{biot_u4},\\
\partial^2_t\v v +\tfrac{\eta}{\rho_e\kappa}\partial_t\v v&=& \v G_2+\mathcal{Q}_2 \label{biot_w4},
\end{eqnarray}
where $\mathcal{Q}_1$ and $\mathcal{Q}_2$ are fist order differential operators that involve $r,s,\v u, \v D$. 
The operator $\mathcal{Q}_1$ also contains $\partial_t\v v$.
Taking the divergence on both sides of the equations (\ref{biot_u3}) and (\ref{biot_w3}) 
and with the help of the equality $\Delta\v u=\nabla r-\curl \v m$, we have
\begin{eqnarray}
\partial^2_t r -(c+a_{11})\Delta r-a_{12}\Delta s
	&=& \div\v G_1+\mathcal{S}_1 \label{biot_p3}, \\
\partial^2_t s -a_{21}\Delta r-a_{22}\Delta s
	&=& \div\v G_2+\mathcal{S}_2 \label{biot_q3},
\end{eqnarray}
where $\mathcal{S}_1$ and $\mathcal{S}_2$ are first order differential operators of $r, s, \v D, \v u, \v m$.
Besides, they also contain $\partial_t\v v$.
Taking the $\curl$ on both sides of the equations (\ref{biot_u3}) and (\ref{biot_w3}) gives
\begin{eqnarray}
\partial^2_t \v m -c\Delta \v m &=& \curl\v G_1+\mathcal{T}_1\label{biot_m},\\
\partial^2_t \v n +\tfrac{\eta}{\rho_e\kappa}\partial_t \v n &=& \curl\v G_2+\mathcal{T}_2\label{biot_n},
\end{eqnarray}
where $\mathcal{T}_1$ and $\mathcal{T}_2$ are first order differential operators of $r, s, \v D, \v u, \v m$.
The expression of $\mathcal{T}_1$ also involves the terms $\partial_t\v v, \partial_t\v n$ and $\mathcal{T}_2$ also contains $\partial_t\v v$.
\vskip .3cm
We emphasize that the presence of the terms $\partial_t\v v$ and $\partial_t \v n$ in the right-hand sides 
$\mathcal{Q}_1$ and $\mathcal{T}_1$ prevents us from using the Carleman estimate in 
\cite{bellassoued2013carleman} directly.
The control of $\mathcal{Q}_1$ and $\mathcal{T}_1$ requires an estimation of 
$\partial_t\v v$ and $\partial_t\v n$. This is actually why we change the variables
from $\v w$ to $\v v$.
Applying Lemma \ref{lemma:scalar_wave_carleman} to (\ref{biot_u4}) and (\ref{biot_m}) yield
\begin{eqnarray}
&    & \int_Q e^{2\tau\varphi}\Big(\tau^3|\v u|^2+\tau |\nabla_{\v x, t}\v u|^2\Big)\nonumber\\
&\leq& C_0\int_Q e^{2\tau\varphi}\big(|\v G_1|^2+|\v D|^2+|r|^2+|\nabla_{\v x,t}r|^2+|s|^2+|\nabla_{\v x,t}s|^2+|\partial_t \v v|^2\big),
	\label{carleman_u}\\
&    & \int_Q e^{2\tau\varphi}\Big(\tau^3|\v m|^2+\tau |\nabla_{\v x, t}\v m|^2\Big)\nonumber\\
&\leq& C_0\int_Q e^{2\tau\varphi}\big(|\curl\v G_1|^2+|\v D|^2+|\nabla\v D|^2+|\v u|^2+|\nabla_{\v x,t}\v u|^2\nonumber\\
&    &	+|r|^2+|\nabla_{\v x,t}r|^2+|s|^2+|\nabla_{\v x,t}s|^2+|\partial_t \v v|^2+|\partial_t \v n|^2\big).\label{carleman_m}
\end{eqnarray}
Further, applying Lemma 2.1 from \cite{bellassoued2013carleman} to (\ref{biot_p3}) and (\ref{biot_q3}), we have
\begin{eqnarray}
 & & \int_Q e^{2\tau\varphi}\Big(\tau^3\big(|r|^2+|s|^2\big)+\tau\big(|\nabla_{\v x, t} r|^2+|\nabla_{\v x, t} s|^2\big)\Big)\nonumber\\
 &\leq& C_0\int_Q e^{2\tau\varphi}\big(|\div\v G_1|^2+|\div\v G_2|^2+|\v D|^2+|\nabla\v D|^2\nonumber \\
 & &+|\v u|^2+|\nabla_{\v x,t}\v u|^2+|\v m|^2+|\nabla_{\v x,t}\v m|^2+|\partial_t \v v|^2\big).\label{carleman_rs}
\end{eqnarray}
Combining (\ref{carleman_u})-(\ref{carleman_rs}) shows that
\begin{eqnarray}
 & & \int_Q e^{2\tau\varphi}\Big(\tau^3\big(|r|^2+|s|^2+|\v u|^2+|\v m|^2\big)+
 	\tau\big(|\nabla_{\v x, t} r|^2+|\nabla_{\v x, t} s|^2+|\nabla_{\v x, t} \v u|^2+|\nabla_{\v x, t} \v m|^2\big)\Big)\nonumber\\
 &\leq& C_0\int_Q e^{2\tau\varphi}\big(|\v G_1|^2+|\nabla\v G_1|^2+|\v G_2|^2+|\nabla\v G_2|^2+|\v D|^2+|\nabla\v D|^2
	+|\partial_t \v v|^2+|\partial_t \v n|^2\big).\label{carleman_biot1}
\end{eqnarray}
Next, we estimate $\partial_t \v v$ and $\partial_t\v n$. Since the differential operator involved 
in (\ref{biot_w4})  acts only  on the variable $t$, we are able to derive the explicit 
 expression of $\partial_t \v v$, and obtain
 %%\partial_t \v v = 0 at t=0% 
\begin{eqnarray*}
\partial_t \v v =  \left(\int_{0}^te^{\int_{0}^s \tfrac{\eta}{\rho_e\kappa}} ( \v G_2+\mathcal{Q}_2 )ds\right)
e^{-\int_{0}^t \tfrac{\eta}{\rho_e\kappa}}.
\end{eqnarray*}
Multiplying both sides by $e^{\tau\varphi}$, and using the fact that $e^{\tau\varphi(\cdot, t)} \leq 
e^{\tau\varphi(\cdot, s)}$  for all $|s| \leq |t| $, we get 
\begin{eqnarray*}
|\partial_t \v v| e^{\tau\varphi} \leq \left(\int_{0}^{|t|}e^{\tau\varphi} e^{\int_{0}^s \tfrac{\eta}{\rho_e\kappa}} | \v G_2+\mathcal{Q}_2|ds\right)
e^{-\int_{0}^t \tfrac{\eta}{\rho_e\kappa}}.
\end{eqnarray*}

Taking the square of the previous relation, integrating over $Q$, and using the H\"older inequality, we finally  find

\begin{eqnarray} \label{ev}
\int_Q  e^{2\tau\varphi}  |\partial_t \v v|^2 \leq C_0
\int_Q e^{2\tau\varphi}
\left( |\v G_2|^2+|\mathcal{Q}_2|^2 \right).
\end{eqnarray}

Proceeding similarly for $\partial_t\v n$, shows that

\begin{eqnarray} \label{en}
\int_{Q}e^{2\tau\varphi} |\partial_t\v n|^2 \leq C_0\int_{Q}e^{2\tau\varphi}
\big(|\curl\v G_2|^2+|\mathcal{T}_2|^2\big)
\end{eqnarray} 
Therefore $\partial_t \v v$ and $\partial_t\v n$ are bounded by $r, s, \v D, \v u, \v m$ and their first order derivatives.
In fact,  multiplying by the weight $e^{2\tau\varphi}$ and integrating over $Q$, deteriorates the stability 
in determining $\partial_t \v v$  and  $\partial_t\v n$ from $r, s, \v D, \v u, \v m$ and their first order derivatives.\\

Considering now the obtained inequalities \eqref{ev} and  \eqref{en},
 the estimate (\ref{carleman_biot1}) becomes
\begin{eqnarray}
 &    & \int_Q e^{2\tau\varphi}\Big(\tau^3\big(|r|^2+|s|^2+|\v u|^2+|\v m|^2\big)+
 	\tau\big(|\nabla_{\v x,t} r|^2+|\nabla_{\v x,t} s|^2+|\nabla_{\v x,t} \v u|^2+
 	|\nabla_{\v x,t} \v m|^2\big)\Big)\nonumber\\
 &\leq& C_0\int_Q e^{2\tau\varphi}\big(|\v G_1|^2+|\nabla\v G_1|^2+|\v G_2|^2+
 	|\nabla\v G_2|^2+|\v D|^2+|\nabla\v D|^2\nonumber\\
 &	  & +|r|^2+|s|^2+|\v u|^2+|\v m|^2+|\nabla_{\v x,t} r|^2+|\nabla_{\v x,t} s|^2+
 	|\nabla_{\v x,t} \v u|^2+|\nabla_{\v x,t} \v m|^2\big).\label{carleman_biot2}
\end{eqnarray}
The lemma is completed by taking $\tau$ large enough to control the zero and first order terms of $r, s, \v u, \v m$ 
on the right hand side of (\ref{carleman_biot2}) and the relations between $\v w,\v G_1,\v G_2$ and $\v u, \v v, \v F_1, \v F_2$.
\end{proof}

Combining Lemma \ref{lemma:em_carleman} and \ref{lemma:biot_carleman}, yields a Carleman estimate for the electroseismic system.
\begin{theorem}\label{theorem:carleman_electroseismic}
Assume that all the parameters in the electroseismic system satisfy the hypotheses of Lemma \ref{lemma:em_carleman} and satisfy
(\ref{lemma:biot_carleman}). Then, there exists a constant $C_0$ such that
\begin{eqnarray*}
 &     & \int_Q e^{2\tau\varphi}\Big(\tau^3\big(|\v D|^2+|\v B|^2+|\v u|^2+|\div\v u|^2+|\div\v w|^2+|\curl\v u|^2\big)\\
 &	   &\quad +\tau  \big(|\nabla_{\v x,t}\v D|^2+|\nabla_{\v x,t}\v B|^2+|\nabla_{\v x,t}\v u|^2+
	|\nabla_{\v x,t}(\div\v u)|^2+|\nabla_{\v x,t}(\div\v w)|^2+|\nabla_{\v x,t}(\curl\v u)|^2\big)\Big) \\
 &\leq & C_0\int_Q e^{2\tau\varphi}\big(|\v F_1|^2+|\v F_2|^2+|\v J_1|^2+|\v J_2|^2+
 	|\nabla\v F_1|^2+|\nabla\v F_2|^2+|\nabla_{\v x,t}\v J_1|^2+|\nabla_{\v x,t}\v J_2|^2\big),
\end{eqnarray*}
for all $\tau$ larger than a positive constant $\tau_0$ and for any $\v D,\v B\in C^2_0(Q)$, $\v u, \v w\in C_0^3(Q)$ 
that solve (\ref{equation:em1}) and (\ref{equation:biot_p1}).

\end{theorem}
%%%%%%%%%%%%%%%%%%%%%%%%%%%%%%%%
\section{The inverse problem}\label{section:inverse_problem}
We now state our main result: a stability theorem for the inverse problem.
\begin{theorem}\label{theorem:inverse_problem}
Let $(\alpha_1, \beta_1, \gamma_1, \xi_1)$ and $(\alpha_2, \beta_2, \gamma_2, \xi_2)$ denote two sets of parameters, which satisfy
the hypotheses of Theorem \ref{theorem:carleman_electroseismic}. Assume that these two sets of parameters coincide in a set 
$\overbar\omega$ where $\omega\subset\Omega$ is a neighborhood of $\partial\Omega$.
Let $(\v D^{(1)}_0, \v B^{(1)}_0)$ and $(\v D^{(2)}_0, \v B^{(2)}_0)$ denote two sets of initial values, such that the matrix
$\v M(\v x)$ defined by
\begin{equation*}\label{equation:M}
\v M(\v x) =
\left(\begin{array}{ccccccc}
\v e_1\times\v B_0^{(1)} & \v e_2\times\v B_0^{(1)} & \v e_3\times\v B_0^{(1)} & -\v D_0^{(1)} & \v 0 & \v 0 & \v 0\\
\v 0 & \v 0 & \v 0 & \v 0 & -\v e_1\times\v D_0^{(1)} & -\v e_2\times\v D_0^{(1)} & -\v e_3\times\v D_0^{(1)} \\
\v e_1\times\v B_0^{(2)} & \v e_2\times\v B_0^{(2)} & \v e_3\times\v B_0^{(2)} & -\v D_0^{(2)} & \v 0 & \v 0 & \v 0\\
\v 0 & \v 0 & \v 0 & \v 0 & -\v e_1\times\v D_0^{(2)} & -\v e_2\times\v D_0^{(2)} & -\v e_3\times\v D_0^{(2)}
\end{array}\right)
\end{equation*} 
has a nonzero $7\times 7$ minor on $\Omega$. Here
$$\v e_1 = \left(\begin{array}{ccc}1 & 0 & 0\end{array}\right),\ \  
\v e_2 = \left(\begin{array}{ccc}0 & 1 & 0\end{array}\right), \ \ 
\v e_3 = \left(\begin{array}{ccc}0 & 0 & 1\end{array}\right).$$
Assuming the following regularity 
$$\v D_k^{(j)},\v B_k^{(j)}\in C^5(Q),\quad \v u_k^{(j)},\v w_k^{(j)}\in C^6(Q)\quad j = 1, 2,$$
of the  solutions to the system (\ref{em_d})-(\ref{em_div}), where 
$\v v_k^{(j)}$ represents the field $\v v$ corresponding to the parameters $(\alpha_k,\beta_k,\gamma_k,\xi_k)$ 
and the $j$-th initial values.
Then, there exist constants $C_0$ and $c_0\in(0,1)$ such that
\begin{equation*}
\int_\Omega\Lambda \le C_0\big(\mathfrak{O}^{(1)}+\mathfrak{O}^{(2)}\big)^{c_0}
\end{equation*}
where 
\begin{eqnarray*}
\Lambda &=& \tilde\Lambda+|\xi|^2+|\nabla\xi|^2,\\
\tilde{\Lambda} &=& |\alpha|^2+|\beta|^2+|\gamma|^2+|\nabla\alpha|^2+|\nabla\beta|^2+|\nabla\gamma|^2+|\nabla\nabla\alpha|^2+|\nabla\nabla\beta|^2,\\
\mathfrak{O}^{(j)} &=& \|\v D^{(j)}\|_{H^4(Q_\omega)}^2+\|\v B^{(j)}\|_{H^4(Q_\omega)}^2+
	\|\v u^{(j)}\|_{H^5(Q_\omega)}^2+\|\v w^{(j)}\|_{H^5(Q_\omega)}^2,
\end{eqnarray*}
and $$\alpha = \alpha_2-\alpha_1,\ \beta=\beta_2-\beta_1,\ \gamma=\gamma_2-\gamma_1,\ \xi=\xi_2-\xi_1,$$
$$\v D= \v D_2-\v D_1,\ \v B= \v B_2-\v B_1,\ \v u= \v u_2-\v u_1,\ \v w= \v w_2-\v w_1 .$$

\end{theorem}
%%%%%%%%%%%%%%%%%%%%%%%%%

\begin{remark}
If we choose $\v B^{(1)}_0=\v e_1$, $\v D^{(1)}_0=\v e_2$, $\v B^{(2)}_0=\v D^{(2)}_0=\v e_3$, the matrix $\v M(\v x)$ formed by
rows (2,3,4,5,8,9,10) and by all the columns of $\v M(\v x)$ is nonsingular. The assumption on the regularity of the
solutions is required to apply the  Carleman estimate to the electroseismic system. 
\end{remark}

\begin{remark}
From the structure of $\v M(\v x)$, the existence of a nonzero $7\times 7$ minor indicates that
there exists a positive constant $c_*$ such that 
$|\v B^{(1)}_0|^2+|\v B^{(2)}_0|^2>c_*$ and $|\v D^{(1)}_0|^2+|\v D^{(2)}_0|^2>c_*$.
\end{remark}

We prove Theorem \ref{theorem:inverse_problem} in 3 steps in the following subsections.

\subsection{A modified Carleman estimate}\label{subsection:modify_carleman}
Since our Carleman estimate is applicable for functions compactly supported in Q,
in the first step we cut off the functions. The near boundary part corresponds to the
measurements and the inner part can be bounded by the Carleman estimate.
The fields $(\v D,\v B,\v u,\v w)$ satisfy the following system of equations in $Q$
\begin{eqnarray}
\partial_t \v D-\curl (\alpha_2\v B) +\gamma_2 \v D &=& \curl(\alpha\v B_1)-\gamma\v D_1 ,\label{em_d_diff}\\
\partial_t \v B+\curl (\beta_2 \v D) &=&  -\curl(\beta\v D_1) , \label{}\\
\rho  \partial^2_t\v u+\rho_f\partial^2_t\v w-\div \v \tau &=& \v 0 \label{} ,\\
\rho_f\partial^2_t\v u+\rho_e\partial^2_t\v w+\nabla p    
	+\tfrac{\eta}{\kappa}\partial_t \v w-\xi_2\v D &=& \xi\v D_1 \label{biot_w_diff},
\end{eqnarray}
with zero initial conditions.
Define $\chi(\v x,t)=\chi_1(\v x)\chi_2(t)$ with $\chi_1\in C^{\infty}_0(\Omega)$, $\chi_2\in C^{\infty}_0(-T,T)$,
$0\le \chi_1,\chi_2 \le 1$ and 
$$\chi_1=1 \text{ in } \overbar{\Omega_0}, \quad \chi_2=1 \text{ in }  [-T+\delta,T-\delta],$$
where $\delta$ is chosen as in Remark \ref{remark:pseudo_convex} and $\Omega_0=\Omega\setminus\overbar{\omega}$.
Denote
$\tilde{\v D}=\chi\v D, \tilde{\v B}=\chi\v B,\tilde{\v u}=\chi\v u,\tilde{\v w}=\chi\v w,$ then
\begin{eqnarray}
\partial_t \tilde{\v D}-\curl (\alpha_2 \tilde{\v B}) +\gamma_2 \tilde{\v D} &=& 
	\chi\big(\curl(\alpha\v B_1)-\gamma\v D_1\big)+\mathcal{P}_1 ,\label{em_d_t}\\
\partial_t \tilde{\v B}+\curl (\beta_2  \tilde{\v D}) &=& -\chi\curl(\beta\v D_1)+\mathcal{P}_2 , \label{}\\
\rho  \partial^2_t\tilde{\v u}+\rho_f\partial^2_t\tilde{\v w}-\div \tilde{\v \tau} &=& \mathcal{P}_3 \label{} ,\\
\rho_f\partial^2_t\tilde{\v u}+\rho_e\partial^2_t\tilde{\v w}+\nabla \tilde{p}    
	+\tfrac{\eta}{\kappa}\partial_t \tilde{\v w}-\xi_2\tilde{\v D} &=& \chi\xi\v D_1+\mathcal{P}_4 \label{biot_w_t},
\end{eqnarray}
where 
$$\mathcal{P}_1=(\partial_t\chi)\v D-\nabla\chi\times(\alpha_2\v B), \quad \mathcal{P}_2=(\partial_t\chi)\v B+\nabla\chi\times(\beta_2\v D),$$
$\mathcal{P}_3,\mathcal{P}_4$ first order differential operators in $\v u, \v w$. Let us note that 
$\mathcal{P}_1, \mathcal{P}_2, \mathcal{P}_3, \mathcal{P}_4$ vanish in 
$Q_0(\delta) = \Omega_0\times(-T+\delta, T-\delta)$. 
Applying Theorem \ref{theorem:carleman_electroseismic} to (\ref{em_d_t})-(\ref{biot_w_t}), we have
\begin{eqnarray}\label{equation:carleman_electroseismic_boundary}
 &     & \int_Q e^{2\tau\varphi}\Big(\tau^3\big(|\tilde{\v D}|^2+|\tilde{\v B}|^2+|\tilde{\v u}|^2\big)
 	+\tau  \big(|\nabla_{\v x, t}\tilde{\v D}|^2+|\nabla_{\v x, t}\tilde{\v B}|^2+|\nabla_{\v x, t}\tilde{\v u}|^2\big)\Big) \nonumber\\
 &\leq & C_0\int_Q e^{2\tau\varphi}\Lambda +C_0\int_{Q_\omega} e^{2\tau\varphi}\Pi+
 	C_0\int_{\Omega\times(-T,-T+\delta)} e^{2\tau\varphi}\Pi+C_0\int_{\Omega\times(T-\delta,T)} e^{2\tau\varphi}\Pi,
\end{eqnarray}
where
\begin{eqnarray*}
\Pi     & = & |\v D|^2+|\v B|^2+|\v u|^2+|\v w|^2+|\nabla_{\v x, t}\v D|^2 +|\nabla_{\v x, t}\v B|^2
   	+|\nabla_{\v x, t}\v u|^2+|\nabla_{\v x, t}\v w|^2+|\nabla\nabla_{\v x, t}\v u|^2+|\nabla\nabla_{\v x, t}\v w|^2.
\end{eqnarray*}
Then from (\ref{equation:carleman_electroseismic_boundary})
and Remark \ref{remark:pseudo_convex}, we have
\begin{eqnarray}\label{equation:modified_carleman}
 &     & \int_{Q_0(\delta)} e^{2\tau\varphi}\Big(\tau^3\big(|{\v D}|^2+|{\v B}|^2+|{\v u}|^2\big)
	+\tau  \big(|\nabla_{\v x, t}{\v D}|^2+|\nabla_{\v x, t}{\v B}|^2+|\nabla_{\v x, t}{\v u}|^2\big)\Big) \nonumber\\
 &\leq & C_0\int_Q e^{2\tau\varphi}\Lambda +C_0 e^{2\tau \Phi} \mathfrak{O}+ C_0 e^{2\tau(1-2\epsilon)}.
\end{eqnarray}
Similarly,  taking the derivative with respect to $t$ on both sides of (\ref{em_d_diff})-(\ref{biot_w_diff}) yields the following inequalities
\begin{eqnarray}\label{equation:modified_carlemanj}
 &     & \int_{Q_0(\delta)} e^{2\tau\varphi}\Big(\tau^3\big(|\partial_t^j{\v D}|^2+|\partial_t^j{\v B}|^2+|\partial_t^j{\v u}|^2\big)
	+\tau  \big(|\nabla_{\v x, t}{\partial_t^j\v D}|^2+|\nabla_{\v x, t}{\partial_t^j\v B}|^2+|\nabla_{\v x, t}{\partial_t^j\v u}|^2\big)\Big) \nonumber\\
 &\leq & C_0\int_Q e^{2\tau\varphi}\Lambda +C_0 e^{2\tau\Phi} \mathfrak{O}+ C_0 e^{2\tau(1-2\epsilon)},
\end{eqnarray}
for $j=1, 2, 3.$
\subsection{Bounding parameters by initial values}\label{subsection:bound_boundary}

Letting $t$ goes to 0 in (\ref{em_d_diff})-(\ref{biot_w_diff}) shows that
\begin{eqnarray}
\partial_t \v D(\v x, 0) &=&  \curl(\alpha\v B_0)-\gamma\v D_0 \label{initial_D},\\
\partial_t \v B(\v x, 0) &=& -\curl(\beta\v D_0) , \label{initial_B} \\
\rho  \partial^2_t\v u(\v x, 0)+\rho_f\partial^2_t\v w(\v x, 0) &=& \v 0 \label{initial_u} ,\\
\rho_f\partial^2_t\v u(\v x, 0)+\rho_e\partial^2_t\v w(\v x, 0) &=& \xi\v D_0 .\label{initial_w}
\end{eqnarray}
Expanding the $\curl$ in (\ref{initial_D}) and (\ref{initial_B}) yields
\begin{eqnarray*}
\nabla\alpha\times\v B_0+\alpha\curl\v B_0-\gamma\v D_0 &=& \partial_t \v D(\v x, 0),\\
-\nabla\beta\times\v D_0-\beta \curl\v D_0 &=& \partial_t \v B(\v x, 0).
\end{eqnarray*}
Substituting (\ref{initial_u}) into (\ref{initial_w}) to eliminate $\v w$ gives
$$\v D_0\xi = -\rho_1\partial_t^2\v u(\v x,0),$$ where $\rho_1 = \tfrac{\rho_0}{\rho_f}$.
Considering the two sets of initial values, we have
\begin{equation}\label{equation:parameter}
\v M(\v x)\left(\begin{array}{c}
\nabla\alpha \\ \gamma \\ \nabla\beta
\end{array}\right) = 
\v N(\v x)\left(\begin{array}{c}
\alpha \\ \beta
\end{array}\right) + \v b(\v x),
\end{equation}
\begin{equation}\label{equation:para_xi}
\v D_0^{(j)}\xi = -\rho_1\partial_t^2\v u^{(j)}(\v x,0),
\end{equation}
where
$$\v N(\v x) = 
\left(\begin{array}{cc}
-\curl\v B_0^{(1)} & \v 0\\
\v 0 & \curl\v D_0^{(1)} \\
-\curl\v B_0^{(2)} & \v 0\\
\v 0 & \curl\v D_0^{(2)}
\end{array}\right),
 \ \
\v b(\v x) = 
\left(\begin{array}{c}
\partial_t\v D^{(1)}(\v x, 0) \\
\partial_t\v B^{(1)}(\v x, 0) \\
\partial_t\v D^{(2)}(\v x, 0) \\
\partial_t\v B^{(2)}(\v x, 0)
\end{array}\right).
$$
Since $\v M(\v x)$ has a $7\times 7$ nonzero minor,
we have 
\begin{equation}\label{equation:bound_alpha_first}
|\nabla\alpha|^2+|\nabla\beta|^2+|\gamma|^2\leq C_0(|\alpha|^2+|\beta|^2+|\v b|^2),
\end{equation}
\begin{equation}\label{equation:bound_xi}
|\xi|^2\le C_0\big(|\partial_t^2\v u^{(1)}(\v x,0)|^2+|\partial_t^2\v u^{(2)}(\v x,0)|^2\big).
\end{equation}
Taking the derivative with respect to the variable $x_k$ on both sides of (\ref{equation:para_xi}), shows that
$$\v D_0^{(j)}\partial_k\xi = -\big(\partial_k\rho_0\big)\partial_t^2\v u^{(j)}(\v x,0)
                              -\rho_0\partial_k\partial_t^2\v u^{(j)}(\v x,0)
                              -\big(\partial_k\v D_0^{(j)}\big)\xi,
$$
and hence
\begin{equation}\label{equation:bound_xi_first}
|\nabla\xi|^2\le C_0\sum_{j=1}^{2}\big(|\partial_t^2\v u^{(j)}(\v x,0)|^2+|\nabla\partial_t^2\v u^{(j)}(\v x,0)|^2\big).
\end{equation}
Therefore
\begin{equation}\label{equation:carleman_xi}
\int_{\Omega} e^{2\tau\varphi_0}\big(|\xi|^2+|\nabla\xi|^2\big)
\le C_0\int_{\Omega_0}e^{2\tau\varphi_0}
	\Big(\sum_{j=1}^{2}\big(|\partial_t^2\v u^{(j)}(\v x,0)|^2+|\nabla\partial_t^2\v u^{(j)}(\v x,0)|^2\big)\Big).
\end{equation}
In addition, taking the derivative with respect to the variable $x_k$ on both sides of (\ref{equation:parameter}), we obtain
\begin{equation*}\label{}
\v M(\v x)\left(\begin{array}{c}
\nabla\partial_k\alpha \\ \partial_k\gamma \\ \nabla\partial_k\beta
\end{array}\right) = 
\partial_k\v N(\v x)\left(\begin{array}{c}
\alpha \\ \beta
\end{array}\right) + 
\v N(\v x)\left(\begin{array}{c}
\partial_k\alpha \\ \partial_k\beta
\end{array}\right) + 
\partial_k\v b(\v x) -
\partial_k\v M(\v x)\left(\begin{array}{c}
\nabla\alpha \\ \gamma \\ \nabla\beta
\end{array}\right),
\end{equation*}
and hence
\begin{equation}\label{equation:bound_alpha_beta_second}
|\nabla\nabla\alpha|^2+|\nabla\nabla\beta|^2+|\nabla\gamma|^2\leq 
C_0(|\alpha|^2+|\beta|^2+|\v b|^2+|\nabla\v b|^2).
\end{equation}
Therefore
\begin{eqnarray}
& & \int_{\Omega} e^{2\tau\varphi_0}\big(|\nabla\nabla\alpha|^2+|\nabla\nabla\beta|^2+|\nabla\gamma|^2\big)\nonumber\\
&\le& C_0\int_{\Omega}e^{2\tau\varphi_0} \big(|\alpha|^2+|\beta|^2\big)
	 	+C_0\int_{\Omega_0}e^{2\tau\varphi_0}\big(|\v b|^2+|\nabla\v b|^2\big),\nonumber\\
&\le& C_0\int_{\Omega}e^{2\tau\varphi_0} \big(|\nabla\alpha|^2+|\nabla\beta|^2\big)
 	+C_0\int_{\Omega_0}e^{2\tau\varphi_0}\big(|\v b|^2+|\nabla\v b|^2\big),\label{equation:para_second_bound}
\end{eqnarray}
because $\alpha, \beta$ are supported in $\Omega_0$.
We recall one lemma from \cite{eller2006carleman}.
\begin{lemma}\label{lemma:lemma1}
There exists constant $\tau_0>0$ and $C_0>0$ such that, for all $\tau>\tau_0$ and $\v v\in H^1_0(\Omega)$, 
$$\tau\int_{\Omega}e^{2\tau\varphi_0}|\v v|^2\le C_0\int_\Omega e^{2\tau\varphi_0}\big(|\curl\v v|^2+|\div\v v|^2\big).$$
\end{lemma}
Applying Lemma \ref{lemma:lemma1} with $\v v=\nabla\alpha$, we have
$$\tau\int_{\Omega}e^{2\tau\varphi_0}|\nabla\alpha|^2\le C_0 \int_\Omega e^{2\tau\varphi_0}|\Delta \alpha|^2
	\leq C_0 \int_\Omega e^{2\tau\varphi_0}|\nabla\nabla \alpha|^2,$$
and hence
$$\tau\int_{\Omega}e^{2\tau\varphi_0}\big(|\nabla\alpha|^2+|\nabla\beta|^2\big)
	\leq C_0 \int_\Omega e^{2\tau\varphi_0}\big(|\nabla\nabla \alpha|^2+|\nabla\nabla\beta|^2\big).$$
Combining with (\ref{equation:bound_alpha_first}) and (\ref{equation:para_second_bound}), we finally obtain the bound
\begin{equation}\label{equation:bound_Lambda_tilde}
\int_{\Omega} e^{2\tau\varphi_0}\tilde\Lambda\le C_0\int_{\Omega_0}e^{2\tau\varphi_0}\big(|\v b|^2+|\nabla\v b|^2\big)
\end{equation}
for $\tau$ large enough.

\subsection{End of the proof of Theorem \ref{theorem:inverse_problem}}\label{subsection:completion_proof}

We recall the following lemma from \cite{bellassoued2013carleman}.
\begin{lemma}\label{lemma:lemma2}
There exist constants $\tau_0>0$ and $C_0>0$ such that, for all $\tau>\tau_0$ and $\v v\in C^1(Q_0(\delta))$,
$$\int_{\Omega_0}|\v v(\v x, 0)|^2\le C_0
	\tau\int_{Q_0(\delta)}|\v v(\v x, t)|^2+C_0\tau^{-1}\int_{Q_0(\delta)}|\partial_t\v v(\v x, t)|^2.$$We recall that $\Omega_0=\Omega\setminus\overbar{\omega}$ and
$Q_0(\delta)=\Omega_0\times(-T+\delta,T-\delta)$.
\end{lemma}
By taking $\v v=e^{\tau\varphi_0}\partial_t\v D^{(j)}(\v x, 0)$ in the above estimate and invoking (\ref{equation:modified_carleman})
to control the derivatives of $\v D^{(j)}(\v x,0)$, we see that
\begin{eqnarray*}
\int_{\Omega_0}e^{2\tau\varphi_0}|\partial_t\v D^{(j)}(\v x, 0)|^2 &\le& 
	C_0\tau \int_{Q_0(\delta)}e^{2\tau\varphi}|\partial_t\v D^{(j)}|^2
	+C_0\tau^{-1}\int_{Q_0(\delta)}e^{2\tau\varphi}|\partial_t^2\v D^{(j)}|^2\\
&\le& C_0\tau^{-2}\mathfrak{E}^{(j)},
\end{eqnarray*}
from (\ref{equation:modified_carlemanj}), where
$$\mathfrak{E}^{(j)} = \int_Q e^{2\tau\varphi}\Lambda + e^{2\tau\Phi} \mathfrak{O}^{(j)}+ e^{2\tau(1-2\epsilon)}.$$
We proceed similarly with the higher-order derivatives of $\v D^{(j)}(\v x,0)$ and with the other fields, to obtain
\begin{eqnarray*}
\int_{\Omega_0}e^{2\tau\varphi_0}|\partial_k\partial_t\v D^{(j)}(\v x, 0)|^2 &\le& 
	C_0\tau \int_{Q_0(\delta)}e^{2\tau\varphi}|\partial_k\partial_t\v D^{(j)}|^2
	+C_0\tau^{-1}\int_{Q_0(\delta)}e^{2\tau\varphi}|\partial_k\partial_t^2\v D^{(j)}|^2\\
&\le& C_0\mathfrak{E}^{(j)},\\
\int_{\Omega_0}e^{2\tau\varphi_0}|\partial_t\v B^{(j)}(\v x, 0)|^2 &\le& 
	C_0\tau \int_{Q_0(\delta)}e^{2\tau\varphi}|\partial_t\v B^{(j)}|^2
	+C_0\tau^{-1}\int_{Q_0(\delta)}e^{2\tau\varphi}|\partial_t^2\v B^{(j)}|^2\\
&\le& C_0\tau^{-2}\mathfrak{E}^{(j)},\\
\int_{\Omega_0}e^{2\tau\varphi_0}|\partial_k\partial_t\v B^{(j)}(\v x, 0)|^2 &\le& 
	C_0\tau \int_{Q_0(\delta)}e^{2\tau\varphi}|\partial_k\partial_t\v B^{(j)}|^2
	+C_0\tau^{-1}\int_{Q_0(\delta)}e^{2\tau\varphi}|\partial_k\partial_t^2\v B^{(j)}|^2\\
&\le& C_0\mathfrak{E}^{(j)},\\
\int_{\Omega_0}e^{2\tau\varphi_0}|\partial_t^2\v u^{(j)}(\v x, 0)|^2 &\le& 
	C_0\tau \int_{Q_0(\delta)}e^{2\tau\varphi}|\partial_t^2\v u^{(j)}|^2
	+C_0\tau^{-1} \int_{Q_0(\delta)}e^{2\tau\varphi}|\partial_t^3\v u^{(j)}|^2\\
&\le& C_0\tau^{-2}\mathfrak{E}^{(j)},\\
\int_{\Omega_0}e^{2\tau\varphi_0}|\partial_k\partial_t^2\v u^{(j)}(\v x, 0)|^2 &\le& 
	C_0\tau \int_{Q_0(\delta)}e^{2\tau\varphi}|\partial_k\partial_t^2\v u^{(j)}|^2
	+C_0\tau^{-1} \int_{Q_0(\delta)}e^{2\tau\varphi}|\partial_k\partial_t^3\v u^{(j)}|^2\\
&\le& C_0\mathfrak{E}^{(j)}.
\end{eqnarray*}
It follows that
\begin{equation}\label{equation:bound_b_u}
\int_{\Omega_0}e^{2\tau\varphi_0}\big(|\v b|^2+|\nabla\v b|^2\big)\le
C_0\big(\mathfrak{E}^{(1)}+\mathfrak{E}^{(2)}\big),
\end{equation}
\begin{equation}
\int_{\Omega_0}e^{2\tau\varphi_0}
	\Big(\sum_{j=1}^{2}\big(|\partial_t^2\v u^{(j)}(\v x,0)|^2+|\nabla\partial_t^2\v u^{(j)}(\v x,0)|^2\big)\Big)
	\le C_0\big(\mathfrak{E}^{(1)}+\mathfrak{E}^{(2)}\big).
\end{equation}
From (\ref{equation:carleman_xi}) and (\ref{equation:bound_Lambda_tilde}), we infer that
\begin{equation}\label{bound_parameters}
\int_{\Omega}e^{2\tau\varphi_0}\Lambda-C_0\int_Q e^{2\tau\varphi}\Lambda \le 
C_0 e^{2\tau\Phi} (\mathfrak{O}^{(1)}+\mathfrak{O}^{(2)})+ C_0 e^{2\tau(1-2\epsilon)}.
\end{equation}
Since $\varphi-\varphi_0<0$ for $|t|>0$, by choosing $\tau_0$ large enough we can make 
$\int_{-T}^{T}e^{2\tau(\varphi-\varphi_0)}$
so small that for all $\tau>\tau_0$,
$$\int_Q e^{2\tau\varphi}\Lambda=\int_{\Omega} e^{2\tau\varphi_0}\Lambda \int_{-T}^{T}e^{2\tau(\varphi-\varphi_0)}
\ll \int_{\Omega} e^{2\tau\varphi_0}\Lambda.$$
Combining this estimate with (\ref{bound_parameters}) and using the fact that $\varphi_0\ge 1-\varepsilon$, it follows that
$$\int_{\Omega}\Lambda\le e^{-2\tau(1-\varepsilon)}\int_{\Omega}e^{2\tau\varphi_0}\Lambda
	\le C_0 e^{2\tau\Phi} (\mathfrak{O}^{(1)}+\mathfrak{O}^{(2)})+ C_0 e^{-2\tau\epsilon} $$
for all $\tau>\tau_0$. Taking 
$$\tau-\tau_0=\frac{-\ln\big(\mathfrak{O}^{(1)}+\mathfrak{O}^{(2)}\big)}{2(\Phi+\epsilon)}.$$
we finally obtain
\begin{eqnarray*}
& & C_0 e^{2\tau\Phi} (\mathfrak{O}^{(1)}+\mathfrak{O}^{(2)})+ C_0 e^{-2\tau\epsilon} \\ 
&\le & C_0e^{2\tau_0\Phi}e^{2(\tau-\tau_0)\Phi}(\mathfrak{O}^{(1)}+\mathfrak{O}^{(2)})+C_0e^{-2\tau_0\varepsilon}e^{-2(\tau-\tau_0)\epsilon}\\
& =  & C_0(\mathfrak{O}^{(1)}+\mathfrak{O}^{(2)})^{\frac{\epsilon}{\epsilon+\Phi}},
\end{eqnarray*}
which completes the proof.

%%%%%%%%%%%%%%%%
\section{Conclusion}
We presented a complete electroseismic model that describes the coupling phenomenon of the 
electromagnetic waves and seismic waves in fluid immersed porous rock. 
Under some assumptions on the physical parameters, we derived  a H\"older  stability estimate  to the inverse problem of  recovery of the electric parameters
 and the coupling coefficient from  interior measurements near the boundary. How to relax  the constraints on the
 physical parameters  will be  the objective of future works.

%%%%%%%%%%%%%%%%%%%%%%%%%%%%
\section{Acknowledgments}
This work was supported in part by  grant
LabEx PERSYVAL-Lab (ANR-11-LABX- 0025-01) and grant 
ANR-17-CE40-0029 of the French National Research Agency ANR (project MultiOnde).

\bibliographystyle{plain}
\bibliography{refs}

\end{document}